\documentclass[12pt]{article}

\usepackage{amsmath}
\usepackage{amsfonts}
\usepackage{amsthm}
\usepackage{amssymb, color}
\usepackage{graphicx, url}

\newtheorem{theorem}{Theorem}[section]

\newtheorem{lemma}[theorem]{Lemma}
\newtheorem{corollary}[theorem]{Corollary}

\theoremstyle{definition}
\newtheorem{definition}[theorem]{Definition}

\begin{document}

\title{Counting edge-Kempe-equivalence classes for 3-edge-colored cubic graphs}

\author{sarah-marie belcastro and Ruth Haas \\Department of Mathematics and Statistics\\
Smith College,
Northampton, MA 01063 USA\\
{\tt  smbelcas@toroidalsnark.net,    rhaas@smith.edu} \\ \\
}

\maketitle

\begin{abstract}
Two edge colorings of a graph are  {\em edge-Kempe equivalent} if one can be obtained from the other by a series of edge-Kempe switches. This work gives some results for the number of edge-Kempe equivalence classes for cubic graphs. In particular we show every 2-connected planar bipartite cubic  graph has exactly one edge-Kempe equivalence class.   Additionally, we exhibit infinite families of nonplanar bipartite cubic graphs with a range of  numbers of edge-Kempe equivalence classes.  
Techniques are developed that will be useful for analyzing other classes of graphs as well.

\end{abstract}

\section{Introduction and Summary}

Back in the frosts of time, Alfred Bray Kempe introduced the notion of changing colorings by switching maximal two-color chains of vertices (for vertex colorings) \cite{kempe} or edges (for edge colorings).  The maximal two-color chains are now called \emph{Kempe chains} and \emph{edge-Kempe chains} respectively; switching the colors along such a chain is called a \emph{Kempe switch} or \emph{edge-Kempe switch} as appropriate. This process is of interest across the study of colorings.  It is also of interest in statistical mechanics, where certain dynamics in the antiferromagnetic $q$-state Potts model correspond to Kempe switches on vertex colorings \cite{phys1}, \cite{phys2}.  
In some cases, these dynamics also correspond to edge-Kempe switches \cite{ms}.

In the present work we are concerned with understanding when two edge-colorings are equivalent under a sequence of edge-Kempe switches and when not.  We allow multiple edges on our (labeled) graphs; loops are prohibited (and will mostly  be excluded by other  constraints such as 3-edge colorability).

 A single edge-Kempe switch is denoted by $-$.  That is,  if coloring $c_i$ becomes coloring $c_j$ after a single edge-Kempe switch, then $c_i - c_j$. If  coloring $c_j$  can be converted to coloring $c_k$ by a sequence of edge-Kempe switches, then $c_j$ and $c_k$ are equivalent; we denote this by $c_j\sim c_k$.  Because $\sim$ is an equivalence relation, we may consider the equivalence classes on the set of colorings of a graph $G$ edge-colored with $n$ colors. In this paper we focus on the \emph{number} of edge-Kempe equivalence classes and denote this quantity by $K'(G,n)$. (In other work this has been denoted Ke$(L(G),n)$ \cite{mohar} and $\kappa_E(G,n)$ \cite{mms}.) 
 
 Note that any global permutation of colors can be achieved by edge-Kempe switches because  the symmetric group $S_n$ is generated by transpositions.  Thus two colorings that differ only by a permutation of colors are  edge-Kempe equivalent.

Recall that $\Delta(G)$ is the largest vertex degree in $G$ and that $\chi'(G)$ is the smallest number of colors needed to properly edge-color $G$. When more colors are used than possibly needed to edge-color the graph, then there is but a single edge-Kempe equivalence class, i.e., when $n> \chi'(G)+1$ then $K'(G,n)=1$ \cite[Thm. 3.1]{mohar}.  More is known if $\Delta(G)$ is restricted; when $\Delta(G)\leq 4, K'(G,\Delta(G)+2)=1$ \cite[Thm. 2]{mms} and when  $\Delta(G)\leq 3, K'(G,\Delta(G)+1)=1$ \cite[Thm. 3]{mms}.  For bipartite graphs there is a stronger result: when $n>\Delta(G)$, $K'(G,n)=1$ \cite[Thm. 3.3]{mohar}.   Little is known about $K'(G, \Delta(G))$.

This paper focuses on  cubic graphs, particularly those that are 3-edge colorable. Mohar suggested  classifying cubic bipartite graphs with $K'(G,3)=1$ \cite{mohar}; we provide a partial answer here.  Mohar  also points out in \cite{mohar} that it follows from a result of Fisk in \cite{fisk} that every planar 3-connected cubic bipartite graph $G$ has $K'(G,3)=1$. We show (in Section \ref{sec:comp}) that for $G$ planar, bipartite, and cubic, $G$ has $K'(G,3)=1$.

The remainder of the paper proceeds as follows. Section \ref{sec:decomp}  introduces decompositions of cubic graphs along  2- or 3-edge cuts that preserve planarity and bipartiteness.  The  theorems in Section \ref{sec:mmcol}  use the edge-cut decompositions to combine and decompose 3-edge colorings. We also show that any edge-Kempe equivalence can avoid color changes at a particular vertex.  Then, in Section \ref{sec:comp} we compute $K'(G,3)$ in terms of the edge-cut decomposition of $G$, and exhibit infinite families of simple nonplanar bipartite cubic graphs with a range of numbers of edge-Kempe equivalence classes.  

\section{Decompositions of Cubic Graphs}\label{sec:decomp}

Any   3-edge cut of a cubic graph may be used to decompose a cubic graph $G$ into two cubic graphs $G_1, G_2$ as follows.  For 3-edge cut  $E_C=\{(s_{11}s_{21}), (s_{12}s_{22}), (s_{13}s_{23})\} $ where vertices $s_{1j}$ are on one side of the cut and $s_{2j}$ on the other,  let the induced subgraphs of $G\setminus E_C$ separated by $E_C$ be $G_1^{\prime}, G_2^{\prime}$. Then for $i=1,2$ define $G_i$ by $V(G_i) = V(G_i^{\prime}) \cup v_i$ and $E(G_i)= E(G_i^{\prime}) \cup  E_{C_i}$ where $E_{C_i} = \{(v_i s_{ij})|\ j=1,2,3\}$,
as is shown in Figure \ref{fig:3-edge-cut}. This decomposition will be written as $G = G_1\thinspace\includegraphics[scale=.3]{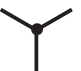}\thinspace G_2$.

\begin{figure}[htp]
\begin{center}
\includegraphics[scale=.7]{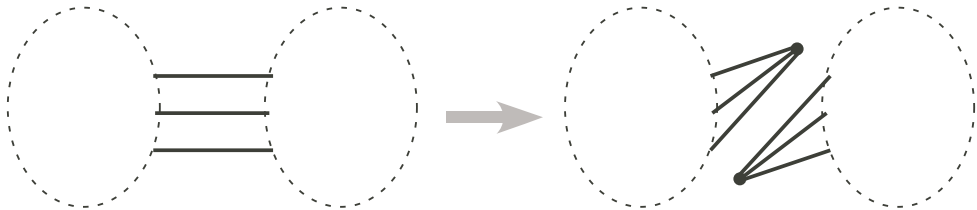}
\caption{Decomposing a graph over a 3-edge cut.}
\label{fig:3-edge-cut}
\end{center}
\end{figure}

A similar decomposition is defined analogously for a 2-edge cut of a cubic graph. Here $G$ has 2-edge cut $E_C= \{ (s_{11}s_{21}), (s_{12}s_{22})\}$ and for $i=1,2$ we define $G_i$ by $V(G_i) = V(G_i^{\prime})$ and 
 $E(G_i)= E(G_i^{\prime}) \cup  e_i$ where $e_i = \{(s_{i1} s_{i2})\}$.   This decomposition will be written as $G=G_1\thinspace\includegraphics[scale=.3]{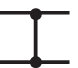}\thinspace G_2$. 
 
 For both of these decompositions,  we say  the edge cut is nontrivial if  both $G_1$ and $G_2$ have fewer vertices than $G$.  Using nontrivial edge cuts, we may decompose a cubic graph $G$ into a set of smaller graphs $\{ G_i\}$ where each $G_i$ has no nontrivial edge cuts (but may have additional multiple edges).  

Notice that these decompositions are reversible, though not uniquely so.
Consider two cubic graphs $G_1, G_2$. Form $G_1\thinspace\includegraphics[scale=.3]{Y.eps}\thinspace G_2$ by distinguishing a vertex on each ($v_1, v_2$ respectively) and identifying the edges incident to 
$v_1$ with the edges incident to $v_2$. \emph{A priori}, there are many ways 
to choose $v_1, v_2$ and many ways to identify their incident edges.   We will abuse the notation $G_1\thinspace\includegraphics[scale=.3]{Y.eps}\thinspace G_2$  by using it to denote a 
particular one of these many choices. Similarly, $G_1\thinspace\includegraphics[scale=.3]{H.eps}\thinspace G_2$ can be formed by choosing an   edge  $e_i= (s_{i1}s_{i2})$ from each $G_i$, deleting $e_i$, and then adding the edges $\{ (s_{11}s_{21}), (s_{12}s_{22})\}$.  Note that constructing  $G_1\thinspace\includegraphics[scale=.3]{H.eps}\thinspace G_2$  is equivalent to cutting an edge of $G_2$ and inserting it into a single edge of $G_1$.

\begin{lemma}\label{lem:plan}  Let $G$ be a cubic graph.  If $G=G_1\thinspace\includegraphics[scale=.3]{Y.eps}\thinspace G_2$ or $G=G_1\thinspace\includegraphics[scale=.3]{H.eps}\thinspace G_2$, then $G$ is planar if and only if $G_1$ and $G_2$ are planar.
\end{lemma}

\begin{proof}  
Suppose that $G$ has a cellular embedding on the sphere.  Then the removal of an edge cut $E_C$ separates $G$ into two subgraphs, $G_1^{\prime}, G_2^{\prime}$  embedded on the sphere, each of which is contained in one of two disjoint discs $D_1,D_2$. 
Note that the resulting  degree-1 and degree-2 vertices of each subgraph are on its outer face (relative to $D_i$) 
 as in Figure \ref{fig:discs}.
\begin{figure}[htp]
\begin{center}
\includegraphics[scale=.7]{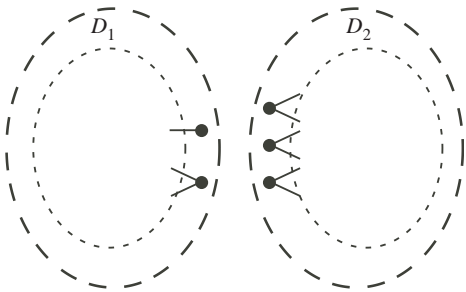}
\caption{ A sample configuration of planar $G_1^{\prime}, G_2^{\prime}$.}
\label{fig:discs}
\end{center}
\end{figure}
If $E_C$ was a 2-edge cut,  edges may be added on the outside face that join these vertices to create planar $G_i$.  If $E_C$ was a 3-edge cut, 
  add  vertices $v_1, v_2$  on the outside faces of discs $D_1, D_2$ respectively, and join $v_i$ to the degree-1 and degree-2 vertices in $D_i$ to create planar $G_i$.

Conversely,   
spherical embeddings of $G_1$ and $G_2$ may be converted to planar drawings with distinguished vertices $v_1, v_2$  or edges $e_1, e_2$ on the outside faces of discs $D_1, D_2$ respectively. 
Removing $v_1, v_2$ (resp. $e_1, e_2$) produces $G$ with three edges (resp. two edges) of a cut missing.  Any desired pairing of the vertices may be completed on a sphere without edges crossing by using judicious placement of $D_i$ (and perhaps flipping one over). This will result in $G_1 \thinspace\includegraphics[scale=.3]{Y.eps}\thinspace G_2$ (resp. $G_1\thinspace\includegraphics[scale=.3]{H.eps}\thinspace G_2$).


\end{proof}

\begin{lemma}\label{lem:bip} Let $G$ be a cubic graph.  If  $G=G_1\thinspace\includegraphics[scale=.3]{H.eps}\thinspace G_2$, or $G=G_1\thinspace\includegraphics[scale=.3]{Y.eps}\thinspace G_2$, then $G$ is bipartite if and only if $G_1$ and $G_2$ are bipartite. 

\end{lemma}

\begin{proof}
If $G$ is a cubic  bipartite graph with nontrivial 2-edge cut, then let there be $m_{j}$ vertices from part $j$ on side $1$; if both cut edges emanate from part $1$ then $3m_1-2 = 3m_2$ which is impossible.  Thus each cut edge must emanate from a different part on side $i$ of the cut, so both removing the edge cut and placing edges on each side maintains bipartition.

Suppose $G$ is  a bipartite cubic graph with  nontrivial 3-edge cut  $E_C$ and $G_1^{\prime}, G_2^{\prime}$ the induced subgraphs of $G\setminus E_C$.  For a bipartition of $G$ to descend naturally to bipartitions of $G_1, G_2$, 
the edges of $E_C$ must be incident  only to  vertices in $G_i^{\prime}$ that are  in the same 
part of $G$. Therefore, assume this is not the case and (without loss of generality) that two of the edges of $E_C$ 
are incident to one part of $G_1^{\prime}$ and the remaining edge of $E_C$ is incident to the other part of $G_1^{\prime}$. Let $G_1^{\prime}$ have $m_j$ vertices belonging to part $j$
of $G$. There are $3m_1-1$ edges emanating from part 1 of $G_1^{\prime}$ 
that must be incident to vertices of part 2 of $G_1^{\prime}$. On the other hand, there are $3m_2 - 2$ 
edges emanating from part 2 of $G_1^{\prime}$ that must be incident to vertices in part 1. Thus $3m_1 
- 1 = 3m_2 - 2$, which 
is impossible.


Conversely, if  $G_1, G_2$ are bipartite,  with distinguished $e_1=s_{11}s_{12}, e_2=s_{21}s_{22}$ for the purpose of forming  $G_1\thinspace\includegraphics[scale=.3]{H.eps}\thinspace G_2$, then the bipartition of $G_1$ extends to $G_1\thinspace\includegraphics[scale=.3]{H.eps}\thinspace G_2$ by assigning $s_{12}$ (resp. $s_{22}$) to the opposite part as $s_{11}$ (resp. $s_{21}$).  Similarly, 
if  $G_1, G_2$ are bipartite,  with distinguished $v_1, v_2$ for the purpose of forming  $G_1\thinspace\includegraphics[scale=.3]{Y.eps}\thinspace G_2$, then  use the bipartition of $G_1$ and assign $v_2$ to the opposite part as $v_1$ to induce a bipartition of $G_1\thinspace\includegraphics[scale=.3]{Y.eps}\thinspace G_2$.
\end{proof}

\begin{theorem}\label{thm:2-not-3decomp} 
A cubic graph $H$ that is 2-connected but not 3-connected  may be decomposed  via $\thinspace\includegraphics[scale=.3]{H.eps}\thinspace$ into a set of cubic loopless graphs $\{H_i\}$ where each $H_i$ is  3-connected.
\end{theorem}

\begin{proof}  The proof is inductive on the number of vertices of $H$. Because $H$ is 2-connected but not 3-connected, there exists a 2-vertex separating set. Figure \ref{fig:2-vert-cut} shows the three possible edge configurations for a 2-vertex separating set of a cubic graph, along with (at top) associated 2-edge cuts. 
\begin{figure}[htp]
\begin{center}
\includegraphics[scale=.5]{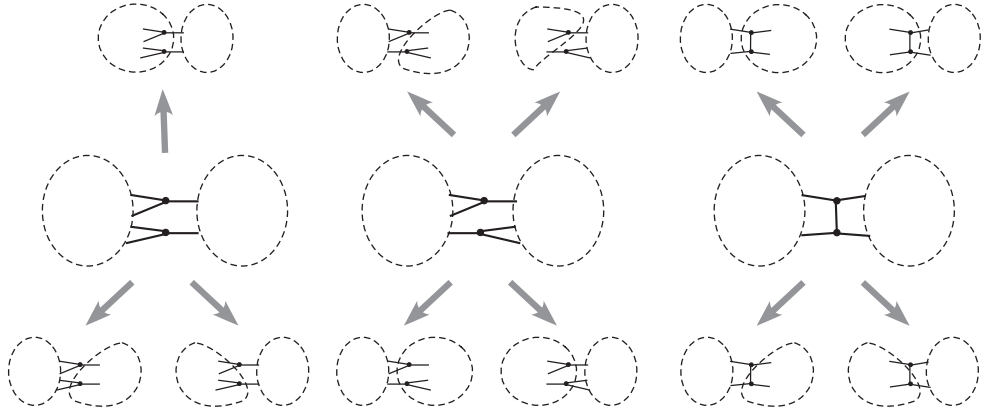} 
\caption{ 2-vertex separating sets with associated 2-edge cuts (top) and 3-edge cuts (bottom).} 
\label{fig:2-vert-cut}
\end{center}
\end{figure}
Each 2-edge cut can be used to form
 $H= H_1\thinspace\includegraphics[scale=.3]{H.eps}\thinspace H_2$, and $\vert H_j\vert < \vert H\vert$ so the inductive hypothesis holds for $H_j$.

\end{proof}


It is worth noting that while the decomposition can create multiple edges, any multiple edge in a cubic graph will be associated
with a 2-edge cut. Thus the final set of $H_j$ will be composed of theta graphs, and graphs with no multiple edges. 

\begin{corollary} The  $\thinspace\includegraphics[scale=.3]{H.eps}\thinspace$ decomposition of 2-connected  cubic graphs  given by Theorem \ref{thm:2-not-3decomp} preserves both planarity and bipartiteness.
\end{corollary}

\begin{proof}
This follows from Lemmas \ref{lem:plan} and \ref{lem:bip}.
\end{proof}

An alternative decomposition using the $\thinspace\includegraphics[scale=.3]{Y.eps}\thinspace$ product can also be found. This is because every 2 vertex separating set is also associated with a 3-edge cut as seen in Figure \ref{fig:2-vert-cut}(bottom).
This decomposition also preserves planarity and bipartiteness.  

\section{Manipulating and Composing Colorings}\label{sec:mmcol}

We begin by showing that we can fix the colors on the edges incident to a given vertex, and  accomplish any sequence of edge-Kempe switches without changing the fixed colors.  As a result, representatives of all edge-Kempe equivalence classes will be present in the set of colorings with fixed colors at a vertex.   The following theorem holds for all base graphs $G$,  not just cubic graphs, and all $n\geq \chi '(G)$. 


\begin{theorem}\label{thm:fix} 
If $c\sim d$ are two  proper edge colorings of a loopless graph $G$, and there exists a vertex $v$ such that $c(e_i) = d(e_i)$ for all $e_i$ incident to $v$, then there exists a sequence of edge-Kempe switches from  $c$ to  $d$   that never change the colors on the edges incident to $v$.
\end{theorem}


Recall that $o_i-o_{i+1}$ is the notation for two colorings that differ by exactly one  edge-Kempe switch. It will be useful to have a further notation for the switch itself. Let $s_i=(\{p_{i_1}, p_{i_2}\}, t_i)$ where $\{p_{i_1}, p_{i_2}\}$ is the pair of colors to be switched on the chain $t_i$  of $G$. Then write  $o_i-_{s_{i}}o_{i+1}$,  if $o_{i+1}$ is obtained from $o_i$ by switching colors $\{p_{i_1}, p_{i_2}\}$ on chain $t_i$. Considering $S_n$ as acting on  the set of colors $\{1,\dots,n\}$, let $\pi_i\in S_n$ be the transposition $\pi_i(p_{i_1}) = p_{i_2}, \pi_i(p_{i_2}) = p_{i_1}$.  

The idea of the proof is as follows.  Each time a switch $s_i= \left(\{p_{i_1},p_{i_2}\}, t_i \right)$ affects an edge incident to $v$,  replace it by making all other $\{p_{i_1},p_{i_2}\}$ switches in the graph. This results in a coloring  of the graph that is equivalent to the original, at the same stage, via a global color permutation. 
Therefore we need to track the colors to be switched on $t_k$, for $k>i$.
Each switch $s_k$ that does not affect an edge incident to vertex $v$ will be replaced by a switch, on the same chain $t_k$, of the colors that are currently on that chain. 
Our proof gives this precisely as an algorithm.

\begin{proof}   Suppose that $c= o_0-_{s_{0}}o_1- _{s_{1}}\dots -_{s_{n-1}}  o_{m}= d$, and there is at least one $i$ such that $v\in t_i$. 
Let $\sigma_0$ be the identity permutation. 
For $0\leq i\leq m-1$, replace $s_i$ with a set of edge-Kempe switches $\hat{s}_i$ as follows.  
Set  
$\hat{\pi}_{i} = \sigma_{i}\pi_i\sigma_{i}^{-1}$ so that $\hat{\pi}_{i}(\sigma_i(p_{i_1})) = \sigma_i(p_{i_2})$. \\
  If $v\not \in t_i$ then  set
   $\hat{s}_i =\left\{ \left(  \{p_{i_1},p_{i_2}\}, t_i\right)\right\}$ and  $\sigma_{i+1}=\sigma_{i}$.  \\
       If $v\in t_i$ then for $\{t_j\}$  the  edge-Kempe  chains of $o_{i}$  in colors $\{p_{i_1},p_{i_2}\}$,
        set $\hat{s}_{i}  = \{( \{\sigma_i(p_{i_1}),\sigma_i(p_{i_2})\}, t_{j})|\ t_{j} \neq t_i\}$  and $\sigma_{i+1}=\sigma_{i}\pi_i$. 
  Note that the  set $\hat{s}_{i}$ may be empty if $t_i$ is the only $\{p_{i_1},p_{i_2}\}$ chain in $o_i$.

Define $\hat{o}_{i+1}$ to be the result of performing the  sets of switches $\hat{s}_1, \dots , \hat{s}_{i}$  to $c$.  
We show that $\hat{o}_{i+1}$ and $o_i$ are equivalent up to a global color permutation  by $\sigma_i$.    Recall that $o_i (e)$ is the color assigned to edge $e$ by $o_i$.  We must show that  on each edge $e$, $\hat{o}_{i+1}(e)= \sigma_{i+1}o_{i+1}(e)$.  We proceed by induction and so assume that for $k\leq i$, $\hat{o}_k(e)=\sigma_ko_{k}(e)$.
 
There are 5 cases. 

\noindent First suppose $v\not \in t_i$.

\smallskip
Case 1a. If  $e\in t_i$ then $\hat{o}_{i+1} (e) = \hat{\pi}_i\hat{o}_i (e) $ because $\hat{\pi}_i$ is the action of switch $\hat{s}_i$.  By definition  of $\hat{\pi}_i$ and using the inductive hypothesis for  $\hat{o}_i$,  $\hat{\pi}_i\hat{o}_i (e)= (\sigma_{i}\pi_i\sigma_{i}^{-1})(\sigma_{i} o_i(e))$. Simplifying, we have $ \sigma_{i}\pi_i o_i(e) = \sigma_{i}o_{i+1}(e)$ (by action of $s_i$ on $o_i$), which, by definition of $\sigma_{i+1}$ in this case, equals $\sigma_{i+1}o_{i+1}(e) $ as desired. Similar reasoning justifies the remaining cases so we present them in an abbreviated fashion.

\smallskip
Case 1b.  If $e\not\in t_i$ then $\hat{o}_{i+1} (e) =\hat{o}_i (e) = \sigma_i o_i(e)=\sigma_{i+1} o_{i+1}(e)$.

\smallskip

\noindent Now suppose $v\in t_i$.

\smallskip
Case 2a. If $o_i(e) \not\in \{p_{i_1},p_{i_2}\}$ then $\hat{o}_{i+1} (e) =\hat{o}_i (e) = \sigma_i o_i(e)=(\sigma_i\pi_i ) o_i(e)
=\sigma_{i+1} o_{i+1}(e)$.

\smallskip
Case 2b. If $o_i(e) \in \{p_{i_1},p_{i_2}\}$  and $e\in t_i$, then the color on $e$ does not change from $\hat{o}_i$ to $\hat{o}_{i+1}$ while it 
did change from $o_i$ to $o_{i+1}$. Thus, 
 $\hat{o}_{i+1} (e) =\hat{o}_i (e) = \sigma _i o_{i}(e)= \sigma _i \pi_i\pi_io_{i}(e)= \sigma _{i+1} o_{i+1}(e)$.
 
 \smallskip
Case 2c. If  $o_i(e) \in \{p_{i_1},p_{i_2}\}$  and $e\not\in t_i$, then the color on $e$ does change from $\hat{o}_i$ to $\hat{o}_{i+1}$ while it 
did  not change from $o_i$ to $o_{i+1}$. Thus, 
 $\hat{o}_{i+1} (e) =\hat{\pi}_i\hat{o}_i (e)= (\sigma_{i}\pi_i\sigma_{i}^{-1})(\sigma_{i} o_i(e))= (\sigma_{i}\pi_i) o_i(e) =\sigma_{i+1} o_{i+1}(e)$.

\smallskip
Finally, we consider $\hat{o}_m$ and compare it to $d$.
Note $c$ and $d$ have the same colors on $v$ by hypothesis, and  the total number of colors used in $d$ is $n$.  If  $n\leq \deg{(v)}+1$, then at most one color is not represented at $v$ and  $ \sigma_m$ must be the identity permutation; thus 
$\hat{o}_m= o_m =d$.
If $n > \deg{(v)}+1$, then it is possible that some colors that do not occur at $v$ are globally permuted  between $o_m$ and  $\hat{o}_n$. 
In this case, additional edge-Kempe switches that globally permute colors can be applied to $\hat{o}_m$ 
so that the coloring now matches $d$.

\end{proof}


This result shows  when counting the number of edge-Kempe equivalence classes  it is sufficient to consider only colorings of $G$ that are different up to global color permutation.   To make this observation precise requires
careful definition of an {\em edge-Kempe-equivalence graph} of a graph. This will be done in \cite{bc-H2}.


\medskip

Returning to cubic graphs, we next consider how combining graphs affects $K'(G,n)$. 
Let  $G_1, G_2$ be two 3-edge-colorable cubic graphs and distinguish a vertex on each ($v_1, v_2$) for the purpose of forming $G_1 \thinspace\includegraphics[scale=.3]{Y.eps}\thinspace G_2$.  Recall that in addition to the choice of $v_1, v_2$, there are multiple ways their incident edges may be identified; by $G_1\thinspace\includegraphics[scale=.3]{Y.eps}\thinspace G_2$ we mean some 
particular set of these choices.  Let $\{x_1,x_2,x_3\}$ and $\{y_1, y_2, y_3\}$ be the ordered sets of edges in $G_1$ and $G_2$ that will be identified in $G_1\thinspace\includegraphics[scale=.3]{Y.eps}\thinspace G_2$.  Similarly, choose a distinguished edge in each graph  ($x\in G_1, y\in G_2$) for the purpose of forming $G_1\thinspace\includegraphics[scale=.3]{H.eps}\thinspace G_2$. 
The following several results relate 3-edge colorings of $G_1$ and $G_2$ to those of $G_1 \thinspace\includegraphics[scale=.3]{Y.eps}\thinspace G_2$ and $G_1\thinspace\includegraphics[scale=.3]{H.eps}\thinspace G_2$. 

\begin{definition}\label{def:xs}
Let $c, d$ be proper edge colorings of $G_1, G_2$ respectively.  There exists a  proper coloring $\hat{d}$ of $G_2$ such that $c(x_i)= \hat{d}(y_i)$ for $i=1, 2, 3$,  and such that $d, \hat{d}$ are the same 
up to a permutation of the colors ($d\sim\hat{d}$). 
Define $(c\thinspace\includegraphics[scale=.3]{Y.eps}\thinspace d)$ to be the  proper coloring of 
$G_1 \thinspace\includegraphics[scale=.3]{Y.eps}\thinspace G_2$ given by\\ $(c\thinspace\includegraphics[scale=.3]{Y.eps}\thinspace d)(e) = \left\{
\begin{array}{c l}
  c(e) & \ {\rm if}\ e\in G_1 \\
  \hat{d}(e) & \ {\rm if}\ e\in G_2 \\
   c(e)=   \hat{d}(e)& \ {\rm if}\ e \ {\rm is\  the\  edge\  resulting \ from \ identifying \ }x_i \ {\rm and \ } y_i. \\
\end{array}
\right.$
Similarly, there exists a proper coloring $\tilde{d}$ of $G_2$ such that $c(x)= \tilde{d}(y)$  and such that $d, \tilde{d}$ are the same 
up to a  global permutation of the colors. 
Define $(c\thinspace\includegraphics[scale=.3]{H.eps}\thinspace d)$ to be the  proper coloring of 
$G_1 \thinspace\includegraphics[scale=.3]{H.eps}\thinspace G_2$ given by\\ $(c\thinspace\includegraphics[scale=.3]{H.eps}\thinspace d)(e) = \left\{
\begin{array}{c l}
  c(e) & \ {\rm if}\ e\in G_1 \\
  \tilde{d}(e) & \ {\rm if}\ e\in G_2 \\
   c(e)=   \tilde{d}(e)& \ {\rm if}\ e \ {\rm is\  one\  of\  the\  edges\ added  \ after \ deleting \ }x\ {\rm and \ } y. \\
\end{array}
\right.$

\end{definition}

Two cases of the 
Parity  Lemma (\cite{Isaacs}) will be useful. 

\begin{lemma}\label{lem:3col} Let $E_C$ be an edge cut of a of a 3-edge-colorable cubic graph $G$ and $c$ be any proper 3-edge coloring of $G$.  Then \\ (a) if $E_C$ is a 2-edge cut, then $c(E_C)$ uses exactly one color, and \\ (b) if $E_C$ is a 3-edge cut, then $c(E_C)$ uses all three colors.
\end{lemma}

%
%

\begin{theorem}\label{thm:coldec} Every 3-edge coloring $f$ of $G = G_1 \thinspace\includegraphics[scale=.3]{Y.eps}\thinspace G_2$  (resp. $G = G_1 \thinspace\includegraphics[scale=.3]{H.eps}\thinspace G_2$) can be written as $c_1 \thinspace\includegraphics[scale=.3]{Y.eps}\thinspace d_1$ (resp. $c_1 \thinspace\includegraphics[scale=.3]{H.eps}\thinspace d_1$) where $c_1$ is some 3-edge coloring of $G_1$ and $d_1$ is some 3-edge coloring of $G_2$.  

\end{theorem}

\begin{proof}  Consider a 3-edge coloring $f$ of $G = G_1 \thinspace\includegraphics[scale=.3]{Y.eps}\thinspace G_2$. There is a 3-edge cut $E_C$ corresponding to the decomposition $G_1 \thinspace\includegraphics[scale=.3]{Y.eps}\thinspace G_2$.  By Lemma \ref{lem:3col}(b), each $e_i\in E_C$ must be a different color in $c$. Therefore considering $f$ on the edges of $G_1$ (and particularly at $v_1$), it is still a proper coloring $c_1$, and likewise $f$ considered on $G_2$ is a proper coloring $d_1$.   The result for $\thinspace\includegraphics[scale=.3]{H.eps}\thinspace$ is similarly an immediate corollary of Lemma \ref{lem:3col}.

\end{proof}

Implicit in the preceding results is the following.

\begin{corollary}   If $G =G_1\thinspace\includegraphics[scale=.3]{Y.eps}\thinspace G_2$  or $G=G_1\thinspace\includegraphics[scale=.3]{H.eps}\thinspace G_2$,  then $G$ is 3-edge colorable if and only if $G_1$ and $G_2$ are 3-edge colorable.
\end{corollary}

Next we note  how edge-Kempe equivalences on the colorings of  $G_1$ and $G_2$ transfer to  edge-Kempe equivalences in combinations of these graphs.

\begin{lemma} \label{lem:star} Let 3-edge colorings $c_1 \sim c_2$ in $G_1$ and  $d_1 \sim d_2$ in $G_2$.  Then $(c_1 \thinspace\includegraphics[scale=.3]{Y.eps}\thinspace d_1) \sim (c_2 \thinspace\includegraphics[scale=.3]{Y.eps}\thinspace d_2)$ in $G_1 \thinspace\includegraphics[scale=.3]{Y.eps}\thinspace G_2$ and $(c_1 \thinspace\includegraphics[scale=.3]{H.eps}\thinspace d_1) \sim (c_2 \thinspace\includegraphics[scale=.3]{H.eps}\thinspace d_2)$ in $G_1 \thinspace\includegraphics[scale=.3]{H.eps}\thinspace G_2$.
\end{lemma}

\begin{proof} Using the notation from Definition \ref{def:xs}, let $c_2'\sim c_2$  by global color permutation such that $c_2'(x_i) = c_1(x_i)$ for $i=1, 2, 3$.  By Theorem \ref{thm:fix},  there exists a sequence of edge-Kempe switches in $G_1$ that exhibits $c_1 \sim c_2'$ and that never changes the color of any edge incident to $v_1$. Similarly, define $\hat{d}'_2\sim\hat{d}_2\sim d_2$ such that there is a sequence of edge-Kempe switches in $G_2$ that exhibits $\hat{d}_1 \sim \hat{d}'_2$ and that never changes the color of any edge incident to $v_2$.  Then $(c_1 \thinspace\includegraphics[scale=.3]{Y.eps}\thinspace d_1)=  (c_1 \thinspace\includegraphics[scale=.3]{Y.eps}\thinspace \hat{d}_1)\sim (c_2' \thinspace\includegraphics[scale=.3]{Y.eps}\thinspace \hat{d}_1) \sim (c_2'\thinspace\includegraphics[scale=.3]{Y.eps}\thinspace \hat{d}_2') \sim (c_2 \thinspace\includegraphics[scale=.3]{Y.eps}\thinspace \hat{d}_2)=(c_2 \thinspace\includegraphics[scale=.3]{Y.eps}\thinspace d_2) $.

For the $\thinspace\includegraphics[scale=.3]{H.eps}\thinspace$ composition, assume without loss of generality that $c_1(x) = d_1(y)$.  Let $c_2''\sim c_2$ by global color permutation such that $c_2''(x)= c_1(x)$ and $d_2''\sim d_2$ by global color permutation such that $d_2''(y)= d_1(y)$. By Lemma \ref{lem:3col}, the two edges created after deleting $x, y$ will be assigned the same color in any proper 3-coloring of $G_1\thinspace\includegraphics[scale=.3]{H.eps}\thinspace G_2$, so fixing the color on one will also fix the color on the other.  Hence,  $(c_1 \thinspace\includegraphics[scale=.3]{H.eps}\thinspace d_1) \sim (c_2'' \thinspace\includegraphics[scale=.3]{H.eps}\thinspace d_1) \sim (c_2''\thinspace\includegraphics[scale=.3]{H.eps}\thinspace d_2'') \sim (c_2 \thinspace\includegraphics[scale=.3]{H.eps}\thinspace d_2)$.

 \end{proof}




\begin{lemma}\label{lem:unstar}  Let $G_1, G_2$ be 3-edge colorable cubic graphs with $G_1 \thinspace\includegraphics[scale=.3]{Y.eps}\thinspace G_2$  and $G_1\thinspace\includegraphics[scale=.3]{H.eps}\thinspace G_2$   particular 
compositions of the two. If $(c_1 \thinspace\includegraphics[scale=.3]{Y.eps}\thinspace d_1) \sim (c_2 \thinspace\includegraphics[scale=.3]{Y.eps}\thinspace d_2)$ in $G_1 \thinspace\includegraphics[scale=.3]{Y.eps}\thinspace G_2$  (resp. $(c_1 \thinspace\includegraphics[scale=.3]{H.eps}\thinspace d_1) \sim (c_2 \thinspace\includegraphics[scale=.3]{H.eps}\thinspace d_2)$ in $G_1 \thinspace\includegraphics[scale=.3]{H.eps}\thinspace G_2$)
then $c_1 \sim c_2$ in $G_1$ and  $d_1 \sim d_2$ in $G_2$. 

\end{lemma}

\begin{proof} 
It is sufficient to show this when  $(c_1 \thinspace\includegraphics[scale=.3]{Y.eps}\thinspace d_1)-_{s}  (c_2 \thinspace\includegraphics[scale=.3]{Y.eps}\thinspace d_2)$ and $(c_1 \thinspace\includegraphics[scale=.3]{H.eps}\thinspace d_1)-_{s}  (c_2 \thinspace\includegraphics[scale=.3]{H.eps}\thinspace d_2)$, where 
$s=(p,t)$ with $p$ a pair of colors and $t$ an edge-Kempe chain. If   $t \subset G_1$ or  $t\subset G_2$, then the lemma holds.  Otherwise,  $t \cap E_C \neq \emptyset$, and 
 $t$ must use exactly 2 edges of $E_C$ because every edge-Kempe chain of a proper 3-edge coloring of a cubic graph is a cycle.  The decomposition  $G_1 \thinspace\includegraphics[scale=.3]{Y.eps}\thinspace G_2$ (resp. $G_1\thinspace\includegraphics[scale=.3]{H.eps}\thinspace G_2$) over $E_C$ will  decompose $t$ into an edge-Kempe chain $t_1$ of $G_1$ and $t_2$ of $G_2$. 
  Then
$c_1 -_{(p,t_1)} c_2$ in $G_1$ and  $d_1- _{(p,t_2)}  d_2$ in $G_2$.
\end{proof}

\begin{theorem}\label{thm:build} Let $G_1, G_2$ be cubic graphs. If $K'(G_1,3)=a$ and $K'(G_2,3)=b$, then $K'(G_1\thinspace\includegraphics[scale=.3]{Y.eps}\thinspace G_2,3)= K'(G_1\thinspace\includegraphics[scale=.3]{H.eps}\thinspace G_2,3)=ab$.
\end{theorem}

\begin{proof} 
Choose colorings $c_1, \dots, c_a$,  one from each of the $a$ edge-Kempe-equivalence classes of $G_1$, and likewise choose colorings $d_1, \dots, d_b$, one from each of  the $b$ edge-Kempe-equivalence classes of $G_2$.  Every 3-edge coloring $f$ of $G_1\thinspace\includegraphics[scale=.3]{Y.eps}\thinspace G_2$ can be written as $f=\hat{c} \thinspace\includegraphics[scale=.3]{Y.eps}\thinspace \hat{d}$ by Theorem \ref{thm:coldec}. $\hat{c} \sim c_i$ for some $c_i \in \{c_1, \dots, c_a\},$ and $\hat{d}\sim d_j$  for some $d_j\in \{d_1, \dots, d_b\}$, so  by Lemma \ref{lem:star}
  $f \sim c_i \thinspace\includegraphics[scale=.3]{Y.eps}\thinspace d_j$ for some $c_i \in \{c_1, \dots, c_a\}, d_j\in \{d_1, \dots, d_b\}$. 
   Further by Lemma \ref{lem:unstar}, 
 $c_{i_1} \thinspace\includegraphics[scale=.3]{Y.eps}\thinspace d_{j_1}\sim c_{i_2} \thinspace\includegraphics[scale=.3]{Y.eps}\thinspace d_{j_2}$ only when  $i_1=i_2, j_1 = j_2$. Therefore 
  there are $ab$ edge-Kempe-equivalence classes of $G_1\thinspace\includegraphics[scale=.3]{Y.eps}\thinspace G_2$.  The proof for $G_1\thinspace\includegraphics[scale=.3]{H.eps}\thinspace G_2$ is identical.
\end{proof}


\section{Results on $K'(G,3)$}\label{sec:comp}


Theorem \ref{thm:build} can be extended to compose several graphs, or alternatively to decompose a graph into many  smaller pieces. We will use the theorem below in both contexts to get results about possible numbers  of edge-Kempe equivalence classes for cubic graphs.

\begin{theorem}\label{thm:triv} Let $G$ be a 3-edge colorable cubic graph.  Then $K'(G,3) = \prod_i K'(G_i,3)$ where $\{G_i\}$ is a decomposition of $G$ along nontrivial  2-edge cuts or 3-edge cuts.
\end{theorem} 

\begin{proof}  This follows from multiple applications of Theorem \ref{thm:build}.
 


\end{proof}

\subsection{Planar, cubic, bipartite graphs}

The following theorem answers a question from \cite[Section 3]{mohar}.

\begin{theorem} Let $H$ be a 2-connected, but not 3-connected, planar bipartite cubic graph.  Then $K'(H,3)=1$.
\end{theorem}

\begin{proof} By Theorem \ref{thm:2-not-3decomp}, $H$ may be decomposed into $\{H_i\}$ where all $H_i$ are 3-connected. 
By Lemmas \ref{lem:plan} and \ref{lem:bip}, all $H_i$ are planar and bipartite. As pointed out in \cite{mohar}, it follows from \cite{fisk} that all 3-connected planar bipartite cubic graphs $G$ have $K'(G,3)=1$  so for all $H_i ,K'(H_i,3)=1$.  It then follows from Theorem \ref{thm:triv} that $K'(H,3)=1$.
\end{proof}

Recall that if $G$ is cubic and bipartite then it must be  bridgeless. Thus we get the following result.

\begin{corollary}  Let $H$ be a planar bipartite cubic graph.  Then $K'(H,3)=1$.
\end{corollary}

\subsection{Nonplanar, cubic, bipartite graphs}

Matters are quite different for \emph{nonplanar} bipartite cubic graphs.  It is well known that $K_{3,3}$ has two  different edge-colorings (shown in Figure \ref{fig:k33col}). 
\begin{figure}[htp]
\begin{center}
\includegraphics[scale=.7]{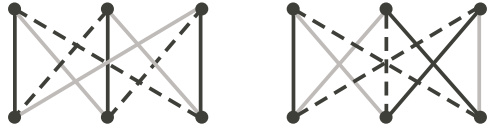}
\caption{The two colorings of $K_{3,3}$.}
\label{fig:k33col}
\end{center}
\end{figure}
In each of these colorings, each color-pair forms a Hamilton cycle.  Therefore, any edge-Kempe switch results in a permutation of the colors and neither coloring of Figure  \ref{fig:k33col} can be obtained from the other.
 Thus, there are two edge-Kempe equivalence classes, i.e. $K'(K_{3,3},3)=2$.

\begin{lemma}\label{lem:snbc}
Every simple bipartite nonplanar cubic graph $B$ with $n\leq 10$  has $K'(B,3)>1$.
\end{lemma}

\begin{proof}
Every simple bipartite nonplanar cubic graph is a subdivision of $K_{3,3}$. 
 To maintain the bipartition  and avoid multiple edges,  $K_{3,3}$ must be subdivided with at least 4 vertices, two on each of two edges.  These edges may be independent or may be incident.
\begin{figure}[htp]
\begin{center}
\includegraphics[scale=.6]{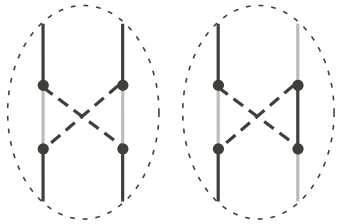}
\caption{The two possible colorings around subdivided  independent or incident edges.}
\label{fig:gad1ind3col}
\end{center}
\end{figure}

Any  coloring of the original graph extends to either one or two new (edge-Kempe equivalent) colorings,  as is shown in Figure \ref{fig:gad1ind3col}.   
If a coloring had three Hamilton cycles before subdivision (as is true for both colorings of $K_{3,3}$), at most it gains an isolated edge-Kempe  cycle after subdivision of this sort.  Thus when subdividing $K_{3,3}$ with a single 4-vertex subdivision, there still exist two colorings that are not edge-Kempe-equivalent. 
\end{proof}

Further examples of nonplanar cubic bipartite graphs with $K'(G,3)>1$ will be given in Section \ref{sec:big}. 
In contrast, Figure \ref{fig:ugh12} shows a bipartite nonplanar cubic graph $U$ with 12 vertices and $K'(U,3)=1$. 
\begin{figure}[htp]
\begin{center}
\includegraphics[scale=.6]{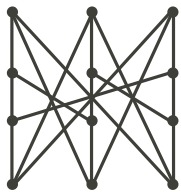}
\caption{A nonplanar bipartite  cubic graph that has a single edge-Kempe equivalence class.}
\label{fig:ugh12}
\end{center}
\end{figure}
$K'(U,3)$ was computed manually and verified using custom \emph{Mathematica} code. We can use $U$ to produce an interesting infinite class of graphs.


\begin{theorem}\label{thm:one}
There exists an infinite family of simple nonplanar  3-connected bipartite cubic graphs $U_k$ with $2+10k$ vertices and $K'(U_k,3)=1$.
\end{theorem}

\begin{proof}
Let $U_k = U \thinspace\includegraphics[scale=.3]{Y.eps}\thinspace\cdots$ ($k$ copies) $\dots \thinspace\includegraphics[scale=.3]{Y.eps}\thinspace U$.  By Theorem \ref{thm:build}, $K'(U_k,3)=1$.  Graphs $U_2, U_3,$ and $U_4$ are shown in Figure \ref{fig:K33gadg}.
\end{proof}
\begin{figure}[htp]
\begin{center}
\includegraphics[scale=.5]{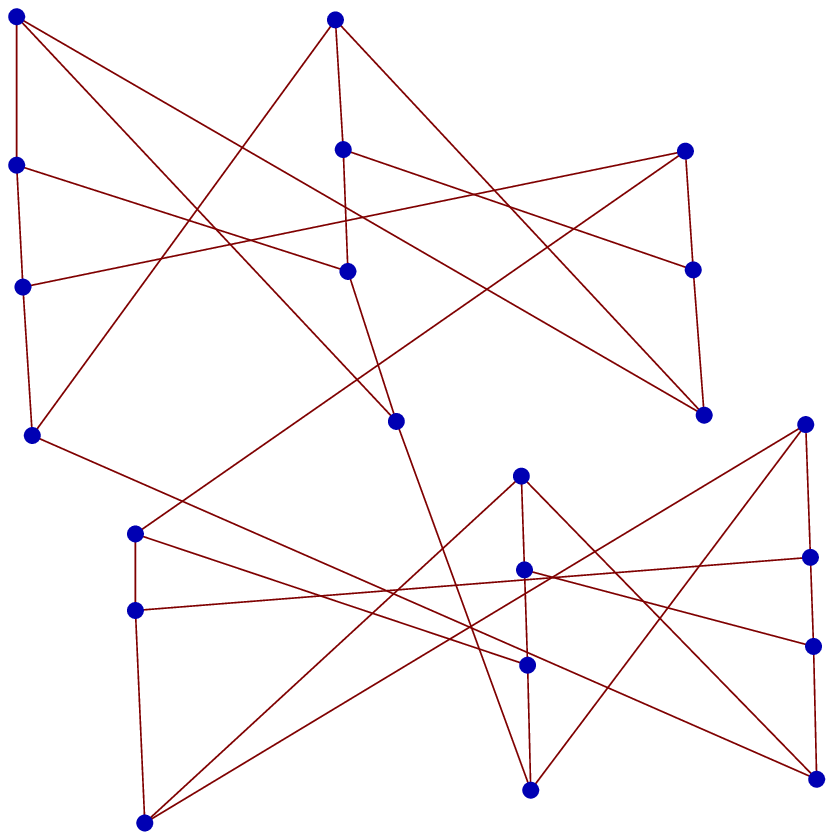} \includegraphics[scale=.5]{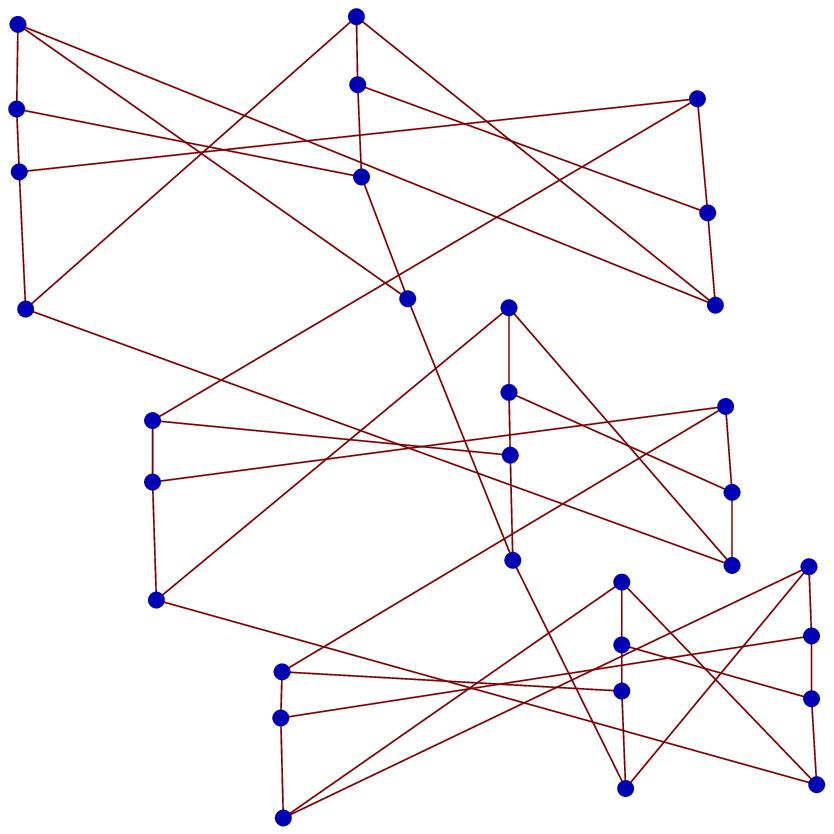} 
\includegraphics[scale=.5]{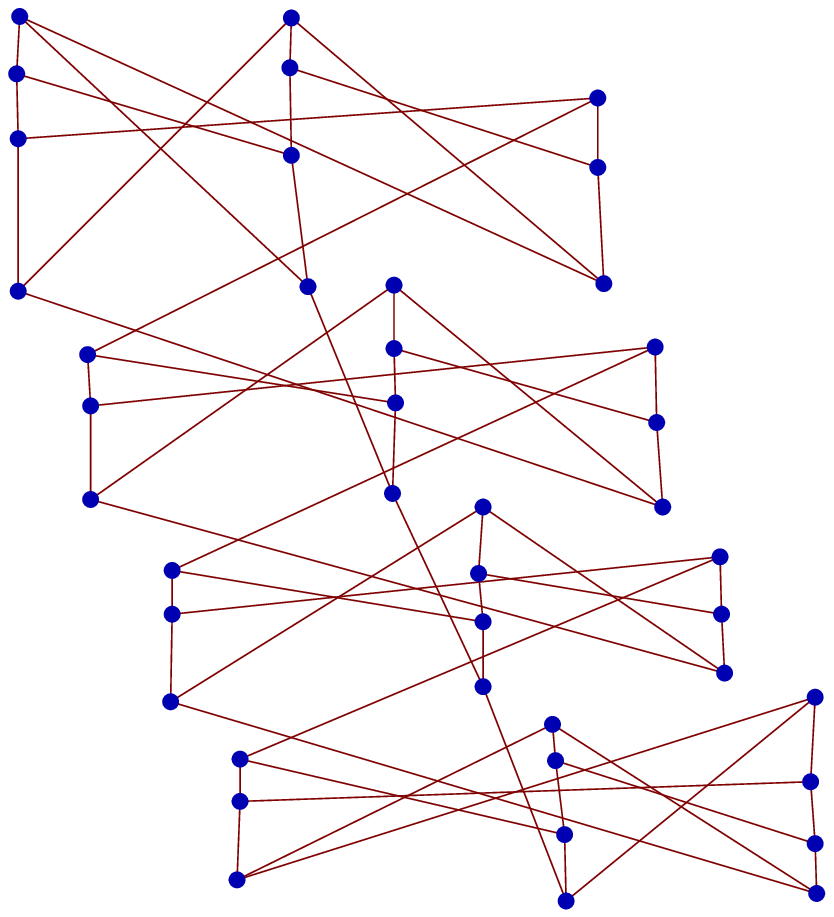}
\caption{Three members of an infinite family of bipartite nonplanar cubic graphs $U_k$, each member of which has a single edge-Kempe equivalence class.}
\label{fig:K33gadg}
\end{center}
\end{figure}

By  $\thinspace\includegraphics[scale=.3]{Y.eps}\thinspace$ composition of $U$ with a planar cubic bipartite graph with $n-10$ vertices we get the following more general result.

\begin{theorem}\label{thm:one-simple}
 For any $n\geq 18$ there is  a simple, nonplanar, bipartite, 3-connected,  cubic graph $G$   with $n$ vertices and $K'(G,3)=1$.
\end{theorem}



Notice that similar results 
can be  obtained for graphs that are only 2-connected as well by using the $\thinspace\includegraphics[scale=.3]{H.eps}\thinspace$ composition.

\subsection{Cubic graphs with $K'(G,3)>1$}\label{sec:big}

We can form $K_{3,3}\thinspace\includegraphics[scale=.3]{Y.eps}\thinspace G$  with any  3-connected cubic graph $G$ to obtain a 3-connected  nonplanar cubic graph.   By Theorem \ref{thm:build}, $$K'(K_{3,3} \thinspace\includegraphics[scale=.3]{Y.eps}\thinspace G, 3) =  K'(K_{3,3},3)K'( G,3) =  2K'( G,3).$$  

  \begin{theorem}\label{thm:2comp} For every even $n\geq 8$, there exists a 3-connected  nonplanar  cubic graph $G$ with $n$ vertices and  exactly $2$ edge-Kempe equivalence classes. 
  \end{theorem}
  \begin{proof} Form $K_{3,3}\thinspace\includegraphics[scale=.3]{Y.eps}\thinspace G$  with any  3-connected planar cubic graph $G$ on $n-4$ vertices to obtain a 3-connected  nonplanar cubic graph with $n$ vertices and $K'(K_{3,3} \thinspace\includegraphics[scale=.3]{Y.eps}\thinspace G, 3) =2$. 
  \end{proof}
  
    \begin{corollary}\label{cor:2comp} For every even $n\geq 12$, there exists a 3-connected  nonplanar bipartite cubic graph $G$ with $n$ vertices and  exactly $2$ edge-Kempe equivalence classes. 
  \end{corollary}
  
    \begin{proof} Form $K_{3,3}\thinspace\includegraphics[scale=.3]{Y.eps}\thinspace G$  with any  3-connected planar cubic bipartite graph $G$ on $n-4$ vertices. The smallest 3-connected planar cubic bipartite graph has 8 vertices.
  \end{proof}
  
   More generally,  once we have one example with $k$  edge-Kempe equivalence classes then there will be an infinite family of them with the same number of classes. 
  
  \begin{theorem}\label{thm:inf}  If $\hat{G}$ is a cubic  graph on $\hat{n}$ vertices with  $k$ edge-Kempe equivalence classes then
  for every even $n\geq \hat{n} + 6$, there exists a cubic graph on $n$ vertices with  exactly $k$ edge-Kempe equivalence classes.  Further, if $\hat{G}$ is planar then 
 a planar  family exists, if $\hat{G}$ is bipartite then 
 a bipartite  family exists and if $\hat{G}$ is 3-connected then a 3-connected family exists. 
  \end{theorem}
  
    \begin{proof} Compose $\hat{G}$  with any cubic planar bipartite graph on $n+2-\hat{n}$ vertices using the $\thinspace\includegraphics[scale=.3]{Y.eps}\thinspace$ operation. The result follows from Theorem \ref{thm:build}.  \end{proof}
  
   We can make graphs with increasingly large numbers of edge-Kempe equivalence classes this way as well.
   
 \begin{theorem}\label{thm:first} For every  $k\geq 1$, there exists a  3-connected nonplanar bipartite cubic graph $G$ with $4k+2$ vertices and $2^{k}$ edge-Kempe equivalence classes. 
 \end{theorem}  
 
 \begin{proof} 
 For $k\geq 1$, take  $K_{3,3}\thinspace\includegraphics[scale=.3]{Y.eps}\thinspace\cdots$ ($k$ copies) $\dots\thinspace\includegraphics[scale=.3]{Y.eps}\thinspace K_{3,3}$, which  has $2+4k$ vertices. By Theorem \ref{thm:build},  it has $2^k$ edge-Kempe equivalence classes. This produces the desired graph. 
 
 
 \end{proof}

 
  \begin{theorem}\label{thm:second} For every simple nonplanar (bipartite) cubic graph $G$ with $n$ vertices, there exists an infinite family of nonplanar (bipartite) cubic graphs $G_k$ such that $G_k$ has $6k+n$ vertices and $2^{k}K'(G,3)$ edge-Kempe equivalence classes. 
 \end{theorem} 
 
 \begin{proof}  Take $G\thinspace\includegraphics[scale=.3]{H.eps}\thinspace K_{3,3} \thinspace\includegraphics[scale=.3]{H.eps}\thinspace\dots\thinspace\includegraphics[scale=.3]{H.eps}\thinspace K_{3,3}$.
\end{proof}





\section{Computations of $K'(G,3)$}

Computing $K'(G,3)$ for particular $G$, or for families of graphs, is surprisingly difficult.  A single computation can be done by brute force by computer, but constructing a proof is another matter.  As examples of the kinds of arguments needed to determine $K'(G,3)$, we analyze   M\"obius ladder graphs, prism graphs, and crossed prism graphs.


\begin{theorem} Let $\mathit{ML}_k$ be the M\"obius ladder graph on $2k$ vertices, let $\mathit{Pr}_k$ be the prism graph on $2k$ vertices, and let $\mathit{CPr}_k$ be the crossed prism graph on $4k$ vertices.  
\begin{enumerate}
\item $K'(\mathit{ML}_k,3)=1$ when $k$ is even and $K'(\mathit{ML}_k,3)=2$ when $k$ is odd.
\item $K'(\mathit{Pr}_k,3)=1$.
\item $K'(\mathit{CPr}_k,3)=1$. 
\end{enumerate}
\end{theorem}
Note that $\mathit{Pr}_k$ is planar, and bipartite exactly when $k$ is even; $\mathit{ML}_k$ is toroidal.

\begin{proof} Our arguments are inductive.

First, consider the edge coloring of $\mathit{ML}_k$ given at left in Figure \ref{fig:mlbox}, and note that it only exists for $k$ odd. Every edge-Kempe chain in this coloring is a Hamilton circuit, so this coloring represents a edge-Kempe-equivalence class of  of $\mathit{ML}_k$.  
\begin{figure}[htp]
\begin{center}
\includegraphics[scale=.6]{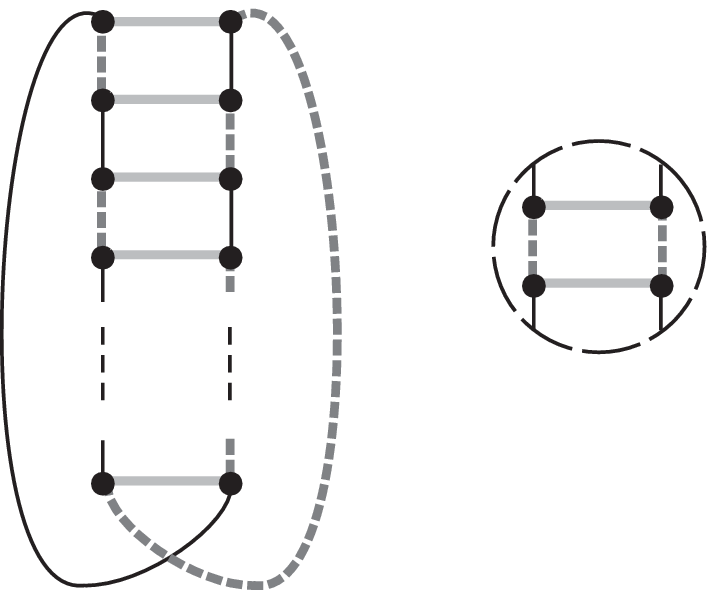} 
\caption{A tri-Hamiltonian edge coloring of $\mathit{ML}_k$ for $k$ odd (left) with a square from some other colorings of $\mathit{ML}_k$ (right).}
\label{fig:mlbox}
\end{center}
\end{figure}
Now consider any other 3-edge coloring of $\mathit{ML}_k$.  If it has a square colored as shown at right in Figure \ref{fig:mlbox}, then the square may be removed (and the remaining half-edges glued together) to produce a 3-edge coloring of $\mathit{ML}_{k-2}$.  If there is no such square in the coloring, then every square must be colored as one of the options shown in Figure \ref{fig:boxes}.
\begin{figure}[htp]
\begin{center}
\includegraphics[scale=.7]{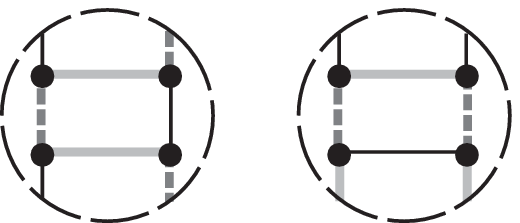} 
\caption{Colorings of squares from $\mathit{ML}_k$ that are edge-Kempe-equivalent to a removable colored square of $\mathit{ML}_k$.}
\label{fig:boxes}
\end{center}
\end{figure}
In either case, we can do a single edge-Kempe switch to produce an edge-Kempe-equivalent coloring  that contains a removable square.  Therefore $K'(\mathit{ML}_k,3) = K'(\mathit{ML}_{k-2},3)$.  To complete the proof, it suffices to show (which direct computation does) that $K'(\mathit{ML}_4,3)=1$  and $K'(\mathit{ML}_3,3)=2$.

Next consider any 3-edge coloring of $\mathit{Pr}_k$.  The same argument as for $\mathit{ML}_k$ applies, so by removing a square we see that $K'(\mathit{Pr}_k,3) = K'(\mathit{Pr}_{k-2},3)$. Because $K'(\mathit{Pr}_3,3)=K'(\mathit{Pr}_4,3)=1$ by direct computation, it then follows that $K'(\mathit{Pr}_k,3)=1$.

Finally, consider any 3-edge coloring of $\mathit{CPr}_k$. Any crossed square must have one of the local colorings shown in Figure \ref{fig:cpbox}.
\begin{figure}[htp]
\begin{center}
\includegraphics[scale=.7]{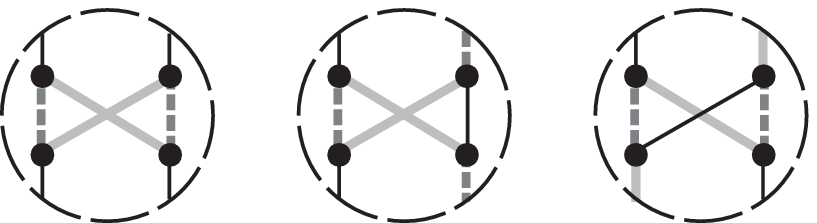} 
\caption{The possible colorings of a crossed square of  $\mathit{CPr}_k$.}
\label{fig:cpbox}
\end{center}
\end{figure}
For the leftmost two colorings of Figure \ref{fig:cpbox}, the crossed square may be removed (and the remaining half-edges glued together) to produce a 3-edge coloring of $\mathit{CPr}_{k-1}$. If there are only crossed squares with coloring type of the rightmost coloring in Figure \ref{fig:cpbox}, we can do a single edge-Kempe switch to produce an edge-Kempe-equivalent coloring that contains a removable crossed square.  (A parity argument shows that there must be at least two edge-Kempe chains in a relevant color pair.) Because $K'(\mathit{CPr}_2,3)=1$ by direct computation, it then follows that $K'(\mathit{CPr}_k,3)=1$.

\end{proof}

\section{Areas for future work}

Two major questions  remain about $K'(G)$ for cubic, nonplanar, bipartite graphs.
First, while we have shown that  there are nonplanar cubic bipartite graphs with $K'(G,3)=1$ and 
also  some with $K'(G,3)> 1$, 
there is as yet no  characterization for when each is true.  
Second,  using \emph{Mathematica} we have found bipartite cubic graphs  where  $K'(G,3) = 1, 2, 3, 4, 6, 8,  9, 15, 17, 35, 131.$  Which natural numbers, and in particular which primes,  $k$ are achievable as $K'(G,3)=k$ for $G$ a cubic nonplanar bipartite 3-connected graph, with no nontrivial edge cuts?
 These same questions can be asked for cubic 3-colorable graphs more generally: which have $K'(G,3) = 1$, and what possible $K'(G,3)$ values can occur?

Beyond just examining the number of edge-Kempe connected components, what is  the structure of the edge-Kempe-equivalence Graph of $G$, whose vertices represent colorings of $G$ and whose edges represent single edge-Kempe switches? This is the topic of
\cite{bc-H2}.

\end{document}